\pdfoutput=1
\documentclass[a4paper,12pt]{article}
\usepackage{amsmath,amssymb,amsthm}
\usepackage{mathtools,mathrsfs}
\usepackage{a4wide,enumerate}
\usepackage{relsize}
\usepackage[numbers,sort&compress]{natbib}
\usepackage{ifpdf}

\ifpdf
\usepackage[bookmarks=false,%
pdfstartview=FitH,linkbordercolor={0.5 1 1},%
citebordercolor={0.5 1 0.5},unicode,%
hyperindex]{hyperref}%
\else
\usepackage{breakurl}                          % For non-PDF mode
\fi

\theoremstyle{plain}% Theorem-like structures
\newtheorem{theorem}{Theorem}[section]
\newtheorem{corollary}[theorem]{Corollary}
\newtheorem{lemma}[theorem]{Lemma}

\theoremstyle{definition}
\newtheorem{example}{Example}

\theoremstyle{remark}

\newcommand{\bbR}{\mathbb{R}}
\newcommand{\setA}{\mathscr{A}}
\newcommand{\setB}{\mathscr{B}}

\newcommand{\Hset}{\mathcal{H}}
\newcommand{\Kset}{\mathcal{K}}
\newcommand{\Mset}{\mathcal{M}}

\DeclareMathOperator*{\co}{co}
\newcommand{\transpose}{\mathsmaller{\mathsf{T}}}% requires relsize.sty

\let\ge\geqslant
\let\le\leqslant

\begin{document}
\date{}
\title{Minimax theorem for the spectral radius\\ of the product of
non-negative matrices}

\author{Victor Kozyakin\thanks{Kharkevich Institute for Information Transmission
Problems, Russian Academy of Sciences, Bolshoj Karetny lane 19, Moscow
127051, Russia, e-mail: kozyakin@iitp.ru\newline\indent~~Kotel'nikov
Institute of Radio-engineering and Electronics, Russian Academy of Sciences,
Mokhovaya 11-7, Moscow 125009, Russia}}

\maketitle

\begin{abstract}
We prove the minimax equality for the spectral radius $\rho(AB)$ of the
product of matrices $A\in\setA$ and $B\in\setB $, where $\setA$ and $\setB$
are compact sets of non-negative matrices of dimensions $N\times M$ and
$M\times N$, respectively, satisfying the so-called hourglass alternative.
\medskip

\noindent\textbf{Keywords:} matrix products; non-negative matrices;
spectral radius; minimax; saddle point
\medskip

\noindent\textbf{AMS Subject Classification:} 15A45; 15B48; 49J35
\end{abstract}

\section{Introduction}\label{S-intro}
In the article, we consider the question about conditions under which, for
compact (closed and bounded) sets of matrices $\setA$ and $\setB$, the
minimax equality holds%
\begin{equation}\label{E:minimax}
\min_{A\in\setA}\max_{B\in\setB}\rho(AB)=\max_{B\in\setB}\min_{A\in\setA}\rho(AB),
\end{equation}
where $\rho(\cdot)$ is the spectral radius of a matrix.

Clearly, equality~\eqref{E:minimax} is not true, in general, see
Example~\ref{Ex1} below. However, some time ago in a private discussion
Eugene Asarin conjectured that equality~\eqref{E:minimax} still might be
valid for certain classes of non-negative matrices. This conjecture, based on
the analysis of properties of the matrix multiplication
games~\cite{ACDDHK:STACS15}, was supported by numerous computer experiments
indicating that equality~\eqref{E:minimax} holds for the classes of matrices
with the so-called `independent row uncertainty'~\cite{BN:SIAMJMAA09} (see
the definition in Section~\ref{S:Hsets}). Unfortunately, attempts to formally
prove the required equality, even for the simplest cases, did not lead to
success for a long time.

The main cause of arising difficulties was the fact that most of the
classical proofs of the minimax theorem for functions $f(x,y)$ assume some
kind of convexity or quasiconvexity in one of the arguments of the function
and concavity or quasiconcavity in the other (see, e.g., \cite{Sion:PJM58}
and also rather old but still urgent survey~\cite{Simons:NOA95}).

As is known, the spectral radius of a matrix has a number of convex-like
properties, see, e.g., \cite{King:QJM61,Fried:LMA81,Elsner:LAA84,Nuss:LAA86}.
In particular, the spectral radius of a nonnegative matrix is both
quasiconvex and quasiconcave with respect to every row of a matrix as well as
to its diagonal elements (but not with respect to the whole matrix). However,
we were not able to find any analogs of convexity/quasiconvexity or
concavity/quasiconcavity of the function $\rho(AB)$ with respect to the
matrix variables $A$ and $B$. Moreover, in view of the identity
$\rho(AB)\equiv\rho(BA)$ the matrices $A$ and $B$ play, in a sense, an
equivalent role in equality~\eqref{E:minimax}. Therefore, any kind of
`convexity' of the function $\rho(AB)$ with respect, say, to the variable $A$
would have to involve its `concavity' with respect to the same variable,
which casts doubt on the applicability of the `convex-concave' arguments in a
possible proof of~\eqref{E:minimax}.

Recently, in~\cite{ACDDHK:STACS15} the author has managed to overcome the
indicated difficulties and to prove equality~\eqref{E:minimax} for the
classes of matrices with independent row uncertainty arising in the theory of
matrix multiplication games~\cite{ACDDHK:STACS15}, the theory of switching
systems~\cite{Koz:ArXiv15-2} and so forth. The relevant proof, as is often
the case, is turned out to be easy enough, and its idea was based on the
so-called hourglass alternative, first formulated in~\cite{ACDDHK:STACS15}
and in a more general form later used in~\cite{Koz:LAA16} for proving the
finiteness conjecture for some classes of non-negative matrices.

The goal of the article is to prove equality~\eqref{E:minimax} for more
general classes of matrices, the so-called classes of non-negative
$\Hset$-sets of matrices resulting from axiomatization of the statements
constituting the hourglass alternative.

The structure of the work is as follows. In Section~\ref{S:Hsets}, we recall
the formulation of the hourglass alternative for the sets of positive
matrices. Then we outline the principal properties of the sets of matrices
satisfying the hourglass alternative, $\Hset$-sets of matrices, among which
the most important property is that the totality of all $\Hset$-sets of
matrices, supplemented by the zero and the identity matrices, forms a
semiring with respect to the Minkowski set operations. In
Section~\ref{S:main}, we show in Theorem~\ref{T:saddle} that the spectral
radius $\rho(AB)$, with the matrices $A$ and $B$ taken from $\Hset$-sets of
matrices, has a saddle point. From this the main result,
Theorem~\ref{T:minimax}, asserting the validity of equality~\eqref{E:minimax}
immediately follows. The proof of Theorem~\ref{T:saddle} is given in
Section~\ref{S:proof:minmax}, its idea is heavily based on the hourglass
alternative.

\section{Hourglass alternative and $\Hset$-sets of matrices}\label{S:Hsets}

Following~\cite{Koz:LAA16}, we recall necessary definitions and facts. For
vectors $x,y\in\bbR^{N}$, we write $x \ge y$ or $x>y$, if all coordinates of
the vector $x$ are not less or strictly greater, respectively, than the
corresponding coordinates of the vector $y$. As usual, a vector or a matrix
is called non-negative (positive) if all its elements are non-negative
(positive).

Denote by $\Mset(N,M)$ the set of all real $(N\times M)$-matrices. This set
can be identified with space $\bbR^{N\times M}$ and therefore, depending on
the context, it can be interpreted as a topological, metric or normed space.
A set of positive matrices $\setA\subset\Mset(N,M)$ will be called
\emph{$\Hset$-set} or \emph{hourglass set} if for each pair $(\tilde{A},u)$,
where $\tilde{A}$ is a matrix from the set $\setA$ and $u$ is a positive
vector, the following assertions hold:
\begin{quote}
\begin{enumerate}[\rm H1:]
\item \emph{either $Au\ge\tilde{A}u$ for all $A\in\setA$ or there exists
    a matrix $\bar{A}\in\setA$ such that $\bar{A}u\le\tilde{A}u$ and
    $\bar{A}u\neq\tilde{A}u$;}

\item \emph{either $Au\le\tilde{A}u$ for all $A\in\setA$ or there exists
    a matrix $\bar{A}\in\setA$ such that  $\bar{A}u\ge\tilde{A}u$ and
    $\bar{A}u\neq\tilde{A}u$.}
\end{enumerate}
\end{quote}

These assertions have a simple geometrical interpretation. Given a matrix
$\tilde{A}\in\setA$  and a vector $u>0$,  imagine that the sets
$\{x:x\le\tilde{A}u\}$ and $\{x:x\ge\tilde{A}u\}$ form the lower and upper
bulbs of an hourglass with the neck at the point $\tilde{A}u$. Let us treat
the elements $Au$ as grains of sand. Then according to assertions H1 and H2
either all the grains $Au$ fill one of the bulbs (upper or lower), or there
remains at least one grain in the other bulb (lower or upper, respectively).
Such an interpretation gives reason to call assertions H1 and H2 the
\emph{hourglass alternative}. This alternative will play a key role in what
follows. It was raised up and used by the author in~\cite{ACDDHK:STACS15} to
analyze the minimax relations between the spectral radii of matrix products,
and also in~\cite{Koz:LAA16} to prove the finiteness conjecture for
non-negative $\Hset$-sets of matrices.

Failure of the inequality $u\ge v$ for vectors $u$ and $v$ does not imply, in
general, the inverse inequality $u\le v$. From this it follows that
assertions H1 and H2 are not valid for arbitrary sets of matrices $\setA$:
non-fulfillment of the inequality $Au\ge\tilde{A}u$ for all $A\in\setA$ does
not mean that for some matrix $\bar{A}\in\setA $ the inverse inequality
$\bar{A}u\le\tilde{A}u$ will be valid. And similarly, non-fulfillment of the
inequality $Au\le\tilde{A}u$ for all $A\in\setA $ does not mean that for some
matrix $\bar{A}\in\setA $ the inverse inequality $\bar{A}u\ge\tilde{A}u$ will
be valid.

In what follows, we will need to make various kinds of limit transitions with
the matrices from the sets under consideration as well as with the sets of
matrices themselves. In this connection, it is natural to restrict our
considerations to compact sets of matrices. By $\Hset(N,M)$ we denote the set
of all compact $\Hset$-sets of positive $(N\times M)$-matrices.

Present some examples of $\Hset$-sets.

\begin{example}
A trivial example of $\Hset$-sets are \emph{linearly ordered} sets of
positive matrices $\setA=\{A_{1}$, $A_{2}$, \ldots, $A_{n}\}$, i.e. the sets
of matrices whose elements satisfy the inequalities
$0<A_{1}<A_{2}<\cdots<A_{n}$. In this case, for each $u>0$, the vectors
$A_{1}u,A_{2}u,\ldots,A_{n}u$ are strictly positive and linearly ordered,
which yields the validity of assertions H1 and H2 for $\setA$. In particular,
any set consisting of a single positive matrix is an $\Hset$-set.\qed%
\end{example}

\begin{example}
A less trivial and more interesting example of $\Hset$-sets, as shown
in~\cite[Lemma 4]{ACDDHK:STACS15} and~\cite[Lemma 1]{Koz:LAA16}, is the class
of sets of positive matrices with independent row uncertainty.
Following~\cite{BN:SIAMJMAA09}, a set of matrices $\setA\subset\Mset(N,M)$ is
called an \emph{IRU-set} (\emph{independent row uncertainty set}) if it
consists of all the matrices
\[
A=\begin{pmatrix}
a_{11}&a_{12}&\cdots&a_{1M}\\
a_{21}&a_{22}&\cdots&a_{2M}\\
\cdots&\cdots&\cdots&\cdots\\
a_{N1}&a_{N2}&\cdots&a_{NM}
\end{pmatrix},
\]
wherein each of the rows $a_{i} = (a_{i1}, a_{i2}, \ldots, a_{iM})$ belongs
to some set of $M$-rows $\setA_{i}$, $i=1,2,\ldots,N$. Clearly, a set
$\setA\subset\Mset(N,M)$ is compact if and only if each set of rows
$\setA_{i}$, $i=1,2,\ldots,N$, is compact.\qed%
\end{example}

To construct another examples of $\Hset$-sets of matrices let us introduce
the operations of Minkowski summation and multiplication for sets of
matrices:
\[
\setA+\setB=\{A+B:A\in\setA,~ B\in\setB\},\quad
\setA\setB=\{AB:A\in\setA,~ B\in\setB\},
\]
and also the operation of multiplication of a set of matrices by a scalar:
\[
t\setA=\setA t=\{tA:t\in\bbR,~A\in\setA\}.
\]
The operation of addition is \emph{admissible} if and only if the matrices
from the sets $\setA$ and $\setB$ are of the same size, while the operation
of multiplication is \emph{admissible} if and only if the sizes of the
matrices from sets $\setA$ and $\setB$ are matched: dimension of the rows of
the matrices from $\setA$ is the same as dimension of the columns of the
matrices from $\setB$. Problems with matching of sizes do not arise when one
considers the sets consisting of square matrices of the same size.

\begin{example}\label{Ex3}
As shown in~\cite[Theorem 2]{Koz:LAA16}, the totality of $\Hset$-sets of
matrices is algebraically closed under the operations of Minkowski summation
and multiplication:
\begin{itemize}
  \item $\setA+\setB\in\Hset(N,M)$ if $\setA,\setB\in\Hset(N,M)$;
  \item $\setA\setB\in\Hset(N,Q)$ if $\setA\in\Hset(N,M)$ and
      $\setB\in\Hset(M,Q)$;
  \item $t\setA=\setA t\in\Hset(N,M)$ if $t>0$ and $\setA\in\Hset(N,M)$.
\end{itemize}
However, in general, $\setA(\setB_{1}+\setB_{2})\neq \setA \setB_{1}
+\setA\setB_{2}$ and $(\setA_{1}+\setA_{2})\setB\neq \setA_{1}\setB
+\setA_{2}\setB$, i.e. the Minkowski operations are not associative. In
particular, $\setA+\setA \neq 2\setA$.

From here it follows that for any integers $n,d\ge1$, the totality
$\Hset(N,M)$ contains all the polynomial sets of matrices
\begin{equation}\label{E:poly}
P(\setA_{1},\setA_{1},\ldots,\setA_{n})=
\sum_{k=1}^{d}\sum_{i_{1},i_{2},\ldots,i_{k}\in\{1,2,\ldots,n\}}
p_{i_{1},i_{2},\ldots,i_{k}}\setA_{i_{1}}\setA_{i_{2}}\cdots\setA_{i_{k}},
\end{equation}
where $\setA_{i}\in\Hset(N_{i},M_{i})$ for $i=1,2,\ldots,n$, and the scalar
coefficients $p_{i_{1},i_{2},\ldots,i_{k}}$ are positive. One must only
ensure that the products $\setA_{i_{1}}\setA_{i_{2}}\cdots\setA_{i_{k}}$
would be admissible and determine the sets of matrices of dimension $N\times
M$.\qed
\end{example}

\subsection{Closure of the set $\Hset(N,M)$}

Given some matrix norm $\|\cdot\|$ on the set $\Mset(N,M)$, denote by
$\Kset(N,M)$ the totality of all compact subsets of $\Mset(N,M)$. Then for
any two sets of matrices $\setA,\setB\in\Kset(N,M)$ the \emph{Hausdorff
metric}
\[
H(\setA,\setB)=
\max\left\{\adjustlimits\sup_{A\in\setA}\inf_{B\in\setB}\|A-B\|,
~\adjustlimits\sup_{B\in\setB}\inf_{A\in\setA}\|A-B\|\right\}
\]
is defined, in which $\Kset(N,M)$ becomes a full metric space. Then
$\Hset(N,M)\subset \Kset(N,M)$, equipped with the Hausdorff metric, also
becomes a metric space.

As is known, see, e.g.,~\cite[Chapter~E, Proposition~5]{Ok07}, any mapping
$F(\setA)$ acting from $\Kset(N,M)$ into itself is continuous in the
Hausdorff metric at some point $\setA_{0}$ if and only if it is both upper
and lower semicontinuous. It is known also~\cite[Section~1.3]{BGMO:84:e} that
the mappings
\[
(\setA,\setB)\mapsto \setA+\setB,\quad (\setA,\setB)\mapsto \setA\setB,\quad
\setA\mapsto\setA\times\setA\times\cdots\times\setA,\quad
\setA\mapsto\co(\setA),
\]
where $\setA$ and $\setB$ are compact sets, are both upper and lower
semicontinuous. Therefore these mappings are continuous in the Hausdorff
metric, and the same continuity properties has any polynomial
mapping~\eqref{E:poly}.

Denote by $\overline{\Hset}(N,M)$ the closure of the set $\Hset(N,M)$ in the
Hausdorff metric. Since the Minkowski summation and multiplication of matrix
sets are continuous in the Hausdorff metric then, as follows from
Example~\ref{Ex3}, all the `polynomial' sets of matrices with the arguments
from $\overline{\Hset}$-sets of matrices (with matched dimensions) take
values again in the $\overline{\Hset}$-set of matrices. However, the answer
to the question when, for a specific $\setA$, the inclusion
$\setA\in\overline{\Hset}(N,M)$ holds, requires further analysis. We restrict
ourselves to the description of only one case where the answer to this
question can be given explicitly~\cite[Lemma~4]{Koz:LAA16}: \emph{the values
of any polynomial mapping~\eqref{E:poly} with the arguments from finite
linearly ordered sets of non-negative matrices or from IRU-sets of
non-negative matrices belong to the closure in the Hausdorff metric of the
totality of positive $\Hset$-sets of matrices.}

\section{Main results}\label{S:main}

In the theory of functions, one of the fundamental criteria of feasibility of
the minimax equality is the following \emph{saddle point principle},
see~\cite[Section 13.4]{von1947theory}.

\begin{lemma}\label{L:MMcriterium}
Let $f (x, y)$ be a continuous function on the product of compact spaces $X
\times Y$. Then
\[
\min_{x}\max_{y} f(x,y) \ge \max_{y}\min_{x} f(x,y).
\]
The exact equality holds if and only if there exists a saddle point, i.e.~a
point $(x_{0}, y_{0})$ satisfying the inequalities
\[
f(x_{0},y) \le f(x_{0},y_{0}) \le f(x,y_{0})
\]
for all $x \in X$, $y \in Y$, and then
\[
\min_{x}\max_{y} f(x,y) =\max_{y}\min_{x} f(x,y)=f(x_{0},y_{0}).
\]
\end{lemma}

This criterion explains the importance of the following saddle point theorem
for the study of the question about minimax equality~\eqref{E:minimax}.

\begin{theorem}\label{T:saddle}
Let $\setA\in\overline{\Hset}(N,M)$ and $\setB\in\overline{\Hset}(M,N)$. Then
there exist matrices $\tilde{A}\in\setA$ and $\tilde{B}\in\setB$ such that
\begin{equation}\label{E:saddleAB}
\rho(\tilde{A}B)\le\rho(\tilde{A}\tilde{B})\le \rho(A\tilde{B})
\end{equation}
for all $A\in\co(\setA)$ and $B\in\co(\setB)$, where $\co(\cdot)$ denotes the
convex hull of a set.
\end{theorem}

In a finite-dimensional space the convex hull of a compact set is a compact
set. Now as the sets of matrices $\setA$ and $\setB$ can be treated as
subsets of finite-dimensional spaces $\bbR^{N\times M}$ and $\bbR^{M\times
N}$, respectively, then the sets $\co(\setA)$ and $\co(\setB)$ in
Theorem~\ref{T:saddle} are compact. If $\setA$ is an IRU-set of matrices
constituted by a set of rows $\setA_{1},\setA_{2},\ldots,\setA_{N}$, then its
convex hull $\co(\setA)$ is the IRU-set constituted by the set of rows
$\co(\setA_{1})$, $\co(\setA_{2})$, \ldots, $\co(\setA_{N})$. If $\setA$ is
an $\Hset$-set of matrices then the structure of the set $\co(\setA)$ is more
complicated.

In Theorem~\ref{T:saddle} the saddle point $(\tilde{A},\tilde{B})$ belongs to
the set $\setA\times\setB$, while the matrices $A$ and $B$, for which the
inequality~\eqref{E:saddleAB} holds, belong to the wider sets:
$(A,B)\in\co(\setA)\times\co(\setB)$. This makes possible deducing a variety
of minimax theorems for the spectral radius $\rho(AB)$ from
Theorem~\ref{T:saddle}.

\begin{theorem}\label{T:minimax}
Let $\setA\in\overline{\Hset}(N,M)$ and $\setB\in\overline{\Hset}(M,N)$. Then
there exists a number $\rho_{*}\ge0$ such that
\[
\min_{A\in\tilde{\setA}}\max_{B\in\tilde{\setB}}\rho(AB)=
\max_{B\in\tilde{\setB}}\min_{A\in\tilde{\setA}}\rho(AB)=\rho_{*}
\]
for any compact sets of matrices $\tilde{\setA}$ and $\tilde{\setB}$
satisfying
\[
\setA\subseteq\tilde{\setA}\subseteq\co(\setA),\quad
\setB\subseteq\tilde{\setB}\subseteq\co(\setB).
\]
\end{theorem}

To prove this theorem, it suffices to note that, by Theorem~\ref{T:saddle}
inequalities~\eqref{E:saddleAB} will take place for all $A\in\tilde{\setA}$
and $B\in\tilde{\setB}$, and then to apply Lemma~\ref{L:MMcriterium}.

Choosing in Theorem~\ref{T:saddle} different sets $\tilde{\setA}$ and
$\tilde{\setB}$, one may obtain a variety of minimax equalities. For example,
putting a $\tilde{\setA}=\setA$ and $\tilde{\setB}=\setB$, we
get~\eqref{E:minimax}. Putting $\tilde{\setA}=\co(\setA)$ and
$\tilde{\setB}=\co(\setB)$, we get another minimax equality:
\[
\min_{A\in\co(\setA)}\max_{B\in\co(\setB)}\rho(AB)=
\max_{B\in\co(\setB)}\min_{A\in\co(\setA)}\rho(AB).
\]
It is worth noting that the minimax value of the spectral radius $\rho(AB)$
in the last equality, and the value of the corresponding minimax in
equality~\eqref{E:minimax} coincide.

The next example demonstrates that Theorem~\ref{T:minimax} is not valid for
general sets of matrices.

\begin{example}\label{Ex1}
Consider the sets of matrices
\[
\setA=\setB=\left\{%
\begin{pmatrix}
1&0\\0&0
\end{pmatrix},~
\begin{pmatrix}
0&0\\0&1
\end{pmatrix}\right\}.
\]
Then
\[
\min_{A\in\setA}\max_{B\in\setB}\rho(AB)=1,\quad
\max_{B\in\setB}\min_{A\in\setA}\rho(AB)=0,
\]
and the minimax equality~\eqref{E:minimax} does not hold in this case.\qed%
\end{example}

The spectral radius $\rho(AB)$ of the product of (rectangular) matrices $A$
and $B$ is not changed by permutation of these matrices and their
transposition. This implies the following corollary.

\begin{corollary}
Theorems~\ref{T:saddle} and~\ref{T:minimax} hold if to replace the function
$\rho(AB)$ by $\rho(BA)$, as well as by $\rho(A^{\transpose}B^{\transpose})$
or by $\rho(B^{\transpose}A^{\transpose})$.
\end{corollary}

\section{Proof of Theorem~\ref{T:saddle}}\label{S:proof:minmax}

Before proceeding to the proof of Theorem~\ref{T:saddle}, we recall some
definitions and establish auxiliary facts.

The \emph{spectral radius} of an $(N\times N)$-matrix $A$ is defined as the
maximal modulus of its eigenvalues and denoted by $\rho(A)$. The spectral
radius depends continuously on the matrix. If $A>0$ then, by the
Perron-Frobenius theorem~\cite[Theorem~8.2.2]{HJ2:e}, the number $\rho(A)$ is
a simple eigenvalue of the matrix $A$, and all the other eigenvalues of $A$
are strictly less than $\rho(A)$ by modulus. The eigenvector
$v={(v_{1},v_{2},\ldots,v_{N})}^{\transpose}$ corresponding to the eigenvalue
$\rho(A)$ (normalized, for example, by the equation
$v_{1}+v_{2}+\cdots+v_{N}=1$) is uniquely determined and positive.

For ease of reference, we summarize some of the well-known statements of the
theory of non-negative matrices, see, e.g.,~\cite[Lemma 2]{Koz:LAA16}
or~\cite[Lemma 3]{ACDDHK:STACS15} for proofs.

\begin{lemma}\label{L:1}
Let $A$ be a non-negative $(N\times N)$-matrix. Then the following assertions
hold:
\begin{enumerate}[\rm(i)]
\item if $Au\le\rho u$ for some vector $u>0$, then $\rho\ge0$ and
    $\rho(A)\le\rho$;
\item moreover, if in conditions of {\rm(i)}  $A>0$ and $Au\neq\rho u$,
    then $\rho(A)<\rho$;
\item if $Au\ge\rho u$ for some non-zero vector $u\ge0$ and some number
    $\rho\ge0$, then $\rho(A) \ge\rho$;
\item moreover, if in conditions of {\rm(iii)} $A>0$ and $Au\neq\rho u$,
    then $\rho(A)> \rho$.
\end{enumerate}
\end{lemma}

To analyze the convergence of sequences $\setA_{n}\to\setA_{\infty}$ in the
Hausdorff metric, it is convenient to use the following lemma, see,
e.g.,~\cite[Chapter~E, Propositions~2,~4]{Ok07}.

\begin{lemma}\label{L:Hausconv}
Let $\setA_{n}\in \Kset(N,M)$ for $n=1,2,\dotsc$. Then
$\setA_{n}\to\setA_{\infty}$ in the Hausdorff metric if and only if the
following assertions are valid:
\begin{enumerate}[\rm(i)]
\item for any sequence of indices $n_{1}<n_{2}<\dotsc$, any sequence of
    matrices $A_{n_{i}}\in\setA_{n_{i}}$, $i=1,2,\dotsc$, contains a
    subsequence converging to some element from~$\setA_{\infty}$;

\item for any matrix $A_{\infty}\in\setA_{\infty}$ and any sequence of
    indices $n_{1}<n_{2}<\dotsc$, there exists a sequence of matrices
    $A_{n_{i}}\in\setA_{n_{i}}$, $i=1,2,\dotsc$, converging
    to~$A_{\infty}$.
\end{enumerate}
\end{lemma}

At last, we will need the following simplified version of Berge's Maximum
Theorem, see~\cite[Ch.~6, \S~3, Theorems 1, 2]{Berge97} and also~\cite[Ch.~E,
Sect.~3]{Ok07}.
\begin {lemma}\label{L:Berge}
If $\varphi$ is a continuous numerical function in the product of topological
spaces $X\times Y$, where $Y$ is compact, then the functions $M(x)=\max_{y\in
Y}\varphi(x,y)$ and $m(x)=\min_{y\in Y}\varphi(x,y)$ are continuous.
\end{lemma}
Note that in the full version of the Maximum Theorem the set over which the
maximum is taken in the definitions of $M(x)$ and $m(x)$ is allowed to vary
with $x$.

We are now ready to prove Theorem~\ref{T:saddle}.

\begin{proof}[Proof of Theorem~\ref{T:saddle}]
First, let $\setA\in\Hset(N,M)$ and $\setB\in\Hset(M,N)$. To construct the
matrices $\tilde{A}\in\setA$ and $\tilde{B}\in\setB$ satisfying
\eqref{E:saddleAB} we proceed as follows. Let us note that for each
$B\in\setB$ there exists a matrix $A_{B}\in\setA$ which minimizes (in
$A\in\setA$) the quantity $\rho(AB)$. Such a matrix $A_{B}$ exists by virtue
of compactness of the set $\setA$ and continuity  of the function $\rho(AB)$
in $A$ and $B$. Then, for each matrix $B\in\setB$, the relations
\[
\rho(A_{B}B)=\min_{A\in\setA}\rho(AB)\le\rho(AB)
\]
will be valid for all $A\in\setA$. Here, by Lemma~\ref{L:Berge} the function
$\min_{A\in\setA}\rho(AB)$ is continuous in $B$, and therefore there exists a
matrix $\tilde{B}\in\setB$ that maximizes its value on the set $\setB$.

Set $\tilde{A}=A_{\tilde{B}}$. In this case
\begin{equation}\label{E:tildeAB}
\max_{B\in\setB}\rho(A_{B}B)=\max_{B\in\setB}\min_{A\in\setA}\rho(AB)=
\min_{A\in\setA}\rho(A\tilde{B})=\rho(A_{\tilde{B}}\tilde{B})=
\rho(\tilde{A}\tilde{B}),
\end{equation}
where the first equality follows from the definition of the matrix $A_{B}$,
the second follows from the definition of $\tilde{B}$, the third follows from
the definition of $A_{\tilde{B}}$, and the fourth follows from the definition
of $\tilde{A}$.

Let $v={(v_{1}, v_{2}, \ldots, v_{N})}^\transpose$ be the positive
eigenvector of the $(N\times N)$-matrix $\tilde{A} \tilde{B}$ corresponding
to the eigenvalue $\rho(\tilde{A}\tilde{B})$ which is uniquely defined up to
a positive factor. By denoting $w = \tilde{B}v\in\bbR^{M}$ we obtain
$\rho(\tilde{A}\tilde{B})v=\tilde{A}w$. Let us show that in this case
\begin{equation}\label{E:rvaaa}
\rho(\tilde{A}\tilde{B})v\le Aw
\end{equation}
for all $A\in\setA$. Indeed, otherwise by assertion H1 of the hourglass
alternative there exists a matrix $\bar{A}\in\setA $ such that
$\rho(\tilde{A}\tilde{B})v\ge\bar{A}w$ and
$\rho(\tilde{A}\tilde{B})v\neq\bar{A}w$ which implies, by definition of the
vector $w$, that $\rho(\tilde{A}\tilde{B})v\ge \bar{A}\tilde{B}v$ and
$\rho(\tilde{A}\tilde{B})v\neq\bar{A}\tilde{B}v$. Then by Lemma~\ref{L:1}
$\rho(\bar{A}\tilde{B})<\rho(\tilde{A}\tilde{B})$, and therefore
$\min_{A\in\setA}\rho(A\tilde{B})< \rho(\tilde{A}\tilde{B})$, which
contradicts to~\eqref{E:tildeAB}. This contradiction completes the proof of
inequality~\eqref{E:rvaaa}.

Similarly, now we show that
\begin{equation}\label{E:wbbb}
w\ge Bv
\end{equation}
for all $B\in\setB$. Again, assuming the contrary by assertion H2 of the
hourglass alternative there exists a matrix $\bar{B}\in\setB$ such that $w\le
\bar{B}v$ and $w\neq \bar{B}v$. This last inequality, together
with~\eqref{E:rvaaa} applied to the matrix $A_{\bar{B}}$, yields
$\rho(\tilde{A}\tilde{B})v\le A_{\bar{B}}\bar{B}v$ and
$\rho(\tilde{A}\tilde{B})v\neq A_{\bar{B}}\bar{B}v$. Then by Lemma~\ref{L:1}
$\rho(\tilde{A}\tilde{B})<\rho(A_{\bar{B}}\bar{B})$, and therefore
$\max_{B\in\setB}\rho(A_{B}B)> \rho(\tilde{A}\tilde{B})$, which again
contradicts to~\eqref{E:tildeAB}. This contradiction completes the proof of
inequality~\eqref{E:wbbb}.

Inequality~\eqref{E:rvaaa} implies, by definition of the vector $w$, that
\[
\rho(\tilde{A}\tilde{B})v\le A\tilde{B}v
\]
for all $A\in\setA$. Then this inequality holds also for all
$A\in\co(\setA)$, which by Lemma~\ref{L:1} yields
\begin{equation}\label{E:propAB2}
\rho(\tilde{A}\tilde{B})\le \rho(A\tilde{B}),\qquad A\in\co(\setA).
\end{equation}

Similarly, left-multiplying the inequality~\eqref{E:wbbb} to the positive
matrix $\tilde{A}$, and taking into account the equality
$\rho(\tilde{A}\tilde{B})v=\tilde{A}w$, we see that
\[
\rho(\tilde{A}\tilde{B})v\ge \tilde{A}Bv
\]
for all $B\in\setB$. Then this inequality holds also for all
$B\in\co(\setB)$, which by Lemma~\ref{L:1} yields
\begin{equation}\label{E:propAB1}
\rho(\tilde{A}B)\le \rho(\tilde{A}\tilde{B}),\qquad B\in\co(\setB).
\end{equation}

Inequalities~\eqref{E:propAB2} and~\eqref{E:propAB1} complete the proof of
the theorem in the case when $\setA\in\Hset(N,M)$ and $\setB\in\Hset(M,N)$.

We proceed to the final stage of the proof. Let now
$\setA\in\overline{\Hset}(N,M)$ and $\setB\in\overline{\Hset}(M,N)$. Then,
for $n=1,2,\dotsc$, there exist sets of matrices $\setA_{n}\in\Hset(N,M)$ and
$\setB_{n}\in\Hset(M,N)$ such that
\begin{equation}\label{E:AnABnB}
\setA_{n}\to\setA,\quad \setB_{n}\to\setB.
\end{equation}
Therefore, as already shown, in virtue of~\eqref{E:propAB2}
and~\eqref{E:propAB1}, for each $n$, there exist matrices
$\tilde{A}_{n}\in\setA_{n}$ and $\tilde{B}_{n}\in\setB_{n}$ such that
\begin{alignat}{2}\label{E:propAB2eps}
\rho(\tilde{A}_{n}\tilde{B}_{n})&\le
\rho(A\tilde{B}_{n}),\qquad
&A\in\co(\setA_{n}),\\
\label{E:propAB1eps}
\rho(\tilde{A}_{n}\tilde{B}_{n})&\ge\rho(\tilde{A}_{n}B),\qquad
&B\in\co(\setB_{n}).
\end{alignat}

By~\eqref{E:AnABnB}, in view of the compactness of the sets $\setA$, $\setB$,
$\setA_{n}$ and $\setB_{n}$, each of the sequences of matrices $\{A_{n}\}$
and $\{B_{n}\}$ without loss of generality may be treated convergent:
$\tilde{A}_{n}\to\tilde{A}$ and $\tilde{B}_{n}\to\tilde{B}$, where due to
assertion (i) of Lemma~\ref{L:Hausconv} $\tilde{A}\in\setA$ and
$\tilde{B}\in\setB$, i.e.
\begin{equation}\label{E:ABtildelim}
\tilde{A}_{n}\to\tilde{A}\in\setA,\quad
\tilde{B}_{n}\to\tilde{B}\in\setB.
\end{equation}

Finally, let us take an arbitrary matrix $A\in\co(\setA)$. Then, by
definition of the convex hull of a set, the matrix $A$ is a finite convex
combination of matrices from $\setA$, i.e.
\[
A=\sum_{i=1}^{r}\lambda_{i}A^{(i)},
\]
where $r$ is some integer, $\lambda_{i}$ are non-negative numbers whose sum
is $1$, and $A^{(i)}\in\setA$ for $i=1,2,\ldots,r$. Then by assertion (ii) of
Lemma~\ref{L:Hausconv} for each $i=1,2,\ldots,r$ there exist sequences of
matrices $\{A^{(i)}_{n}\}$ such that $A^{(i)}_{n}\in\setA_{n}$ for all
$n=1,2,\dotsc$, and $A^{(i)}_{n}\to\setA^{(i)}$. Therefore the matrices
\[
A_{n}=\sum_{i=1}^{r}\lambda_{i}A^{(i)}_{n}\in\co(\setA_{n})
\]
satisfy the limit relation
\begin{equation}\label{E:Aepslim}
A_{n}\to A.
\end{equation}

Now, substituting the matrices $\tilde{A}_{n}$, $\tilde{B}_{n}$ and $A_{n}$
in~\eqref{E:propAB2eps} we obtain
\begin{equation}\label{E:lastineq}
\rho(\tilde{A}_{n}\tilde{B}_{n})\le
\rho(A_{n}\tilde{B}_{n}).
\end{equation}

Taking the limit in~\eqref{E:lastineq}, due to~\eqref{E:ABtildelim}
and~\eqref{E:Aepslim} we obtain the inequality~\eqref{E:propAB2} valid, this
time, in the case when $\setA\in\overline{\Hset}(N,M)$ and
$\setB\in\overline{\Hset}(M,N)$. Similarly we can prove
inequality~\eqref{E:propAB1} in the case when $\setA\in\overline{\Hset}(N,M)$
and $\setB\in\overline{\Hset}(M,N)$.

The proof of Theorem~\ref{T:saddle} is completed.
\end{proof}

\section*{Acknowledgments}
The author is genuinely grateful to Eugene Asarin for numerous inspiring
discussions and constructive criticism.

\section*{Funding}
The work was carried out at the Institute of Radio-engineering and
Electronics, Russian Academy of Sciences, and was supported by the Russian
Science Foundation, Project no.~\mbox{16--11--00063}.

%\bibliographystyle{elsarticle-num}
%\bibliography{kozbib,kozpub}

\end{document}